\documentclass[11pt]{article}%
\usepackage[utf8]{inputenc}
\usepackage[T1]{fontenc}
\usepackage[english]{babel}
\usepackage{indentfirst}
\usepackage{latexsym}
\usepackage[colorlinks=true, bookmarksnumbered=true, bookmarksopen=true,
bookmarksopenlevel=3, pdfstartview=FitH, linkcolor=cyan, pdfmenubar=true,
pdftoolbar=true, bookmarks=true,citecolor=cyan, urlcolor=magenta,
filecolor=cyan,menucolor=black,plainpages=false,pdfpagelabels]{hyperref}
\usepackage[lmargin=2cm,tmargin=2cm,bmargin=2cm,rmargin=2cm]{geometry}
\usepackage{amsbsy, amsmath,amsfonts,amssymb,amsthm}
\usepackage{times}
\usepackage{graphicx}
\usepackage{amsmath}
\usepackage{amsfonts}
\usepackage{amssymb}%
\newtheorem{theorem}{Theorem}[section]

\newtheorem{proposition}[theorem]{Proposition}
\newtheorem{corollary}[theorem]{Corollary}

\numberwithin{equation}{section}
\begin{document}

\title{Complemented subspaces of homogeneous polynomials}
\author{{{Sergio A. Pérez}{\thanks{S.Pérez was supported by CAPES and CNPq, Brazil. (corresponding author) }}}\\{\small  IMECC, UNICAMP} \\{\small {Rua S\'{e}rgio Buarque de Holanda, 651, CEP 13083-859, Campinas-SP,
Brazil.}}\\{\small \texttt{Email:Sergio.2060@hotmail.com}}\\\vspace{-0.2cm}}
\date{}
\maketitle

\begin{abstract}

Let $\mathcal{P}_{K} (^{n}E; F)$ (resp. $\mathcal{P}_{w} (^{n}E; F)$) denote the subspace of all $P\in \mathcal{P}(^{n}E; F)$ which are compact (resp. weakly continuous on bounded sets). We show that if $\mathcal{P}_{K} (^{n}E; F)$ contains an isomorphic copy of $c_{0}$, then $\mathcal{P}_{K} (^{n}E; F)$ is not
complemented in $\mathcal{P}(^{n}E; F)$. Likewise we show that if $\mathcal{P}_{w} (^{n}E; F)$ contains an isomorphic copy of $c_{0}$, then $\mathcal{P}_{w}(^{n}E; F)$ is not complemented in $\mathcal{P}(^{n}E; F)$.

{\small \medskip\noindent\textbf{Keywords:} Banach space, linear operator, compact operator, homogeneous polynomial, complemented subspace, unconditional basis.}

\end{abstract}

\renewcommand{\baselinestretch}{1.1}

\section{Introduction}

The problem of establishing sufficient conditions for the complementation of the subspace of compact linear operators $\mathcal{L}_{K}(E;F)$ in the space $\mathcal{L}(E;F)$ of all continuous linear operators, has been widely studied by many authors. For example, see Kalton \cite{KALTON}, Emmanuelle \cite{EM}, John \cite{KA}, Bator and Lewis \cite{LEW} and Ghenciu \cite{IOANA}, among others.

Emmanuele \cite{EM} and John \cite{KA} showed that if $c_{0}$ embeds in $\mathcal{L}_{K}(E;F)$ then $\mathcal{L}_{K}(E;F)$ is not complemented in $\mathcal{L}(E;F)$ for every $E$ and $F$ infinite dimensional Banach spaces.

John \cite{KA} proved that if $E$ and $F$ are arbitrary Banach spaces and $T: E\rightarrow F$ is a non compact
operator which admits a factorization $T = A\circ B$ through a Banach space
$G$ with an unconditional basis, then the subspace $\mathcal{L}_{K}(E;F)$
of compact operators contains an isomorphic copy of $c_{0}$ and thus $\mathcal{L}_{K}(E;F)$ is not
complemented in $\mathcal{L}(E;F)$.
John \cite{KA} also proved that if $E$ and $F$ are infinite dimensional Banach spaces, such that each
non compact operator $T \in \mathcal{L}(E;F)$ factors through a Banach space $G$ with an unconditional basis, then the following conditions are equivalent:
\begin{enumerate}
\item $\mathcal{L}_{K} (E; F)=\mathcal{L} (E;F)$.
\item $\mathcal{L} (E; F)$ contains no copy of $\ell_{\infty}$.
\item $\mathcal{L}_{K} (E; F)$ contains no copy of $c_{0}$.
\item  $\mathcal{L}_{K} (E; F)$ is complemented in $\mathcal{L} (E; F)$.
\end{enumerate}

Ghenciu \cite{IOANA} obtained the following result:
Let $E$ and $F$ be Banach spaces, and let $G$ be a Banach space with an unconditional basis $(g_{n})$ and coordinate functionals $(g^{\prime}_{n})$.

\begin{enumerate}

\item [(a)] If there exist operators $R\in\mathcal{L}(G; F)$ and $S\in\mathcal{L}(E; G)$ such that $(R(g_{n}))$ is a seminormalized basic sequence
in $F$ and $(S^{\prime}(g^{\prime}_{n}))$ is not relatively compact in $E^{\prime}$, then $\mathcal{L}_{K}(E; F)$ is not complemented in
$\mathcal{L}(E; F)$.

\item [(b)] If there exist operators $R\in\mathcal{L}(G; F)$ and $S\in\mathcal{L}(E; G)$ such that $(R(g_{n}))$ is a seminormalized basic sequence
in $F$ and $(S^{\prime}(g^{\prime}_{n}))$ is not relatively weakly compact in $E^{\prime}$, then $\mathcal{L}_{wK}(E; F)$ is not complemented in
$\mathcal{L}(E; F)$.
\end{enumerate}
This result generalizes results of several authors \cite{EMA},\cite{LEW}, \cite{FEDER}.
In this paper, we obtain polynomial versions of the preceding results.

This paper is based on part of the author's doctoral thesis at the Universidade Estadual de Campinas. This research has been supported by CAPES and CNPq.
The author is grateful to his thesis advisor, Professor Jorge Mujica, for his advice and help.

\section{Preliminaries}

Let $E$ and $F$ denote Banach spaces over $ \mathbb{K}$, where $ \mathbb{K}$ is $ \mathbb{R}$ or  $\mathbb{C}$. Let $E^{\prime}$
denote the dual of $E$. Denote by $ \mathcal{L}(E;F)$, $\mathcal{L}_{K}(E;F)$ and $\mathcal{L}_{wK}(E;F)$, respectively, the spaces of all bounded, all compact and all weakly compact linear operators of $E$ into $F$. Let $\mathcal{P}(^{n}E; F)$ denote the Banach space of all continuous $n$-homogeneous polynomials from $E$ into $F$. We omit $F$ when $F = \mathbb{K}$.
Let $\mathcal{P}_{w}(^{n}E; F)$ denote the subspace of all $P\in\mathcal{P}(^{n}E; F)$
which are weakly continuous on bounded sets, that is
the restriction $P|_{B}: B \rightarrow F $ is continuous for each bounded set $B\subset E$,
when $B $ and $F$ are endowed with the weak topology and the norm topology,
respectively. Let $\mathcal{P}_{K}(^{n}E; F)$ denote the
subspace of all $P\in \mathcal{P}(^{n}E; F)$  which map bounded sets onto relatively compact sets.
Let $\mathcal{P}_{wK}(^{n}E; F)$ denote the
subspace of all $P\in \mathcal{P}(^{n}E; F)$  which map bounded sets onto relatively weakly compact sets.
We always have the inclusions
$$ P_{w}(^{n}E; F)\subset \mathcal{P}_{K}(^{n}E; F)\subset\mathcal{P}_{wK}(^{n}E; F)\subset \mathcal{P}(^{n}E; F).$$
We refer to  \cite{SEAN} or \cite{LMUJICA} for background information on the theory of polynomials on Banach spaces.

$E$ is isomorphic to a complemented subspace of $F$ if and only if there are $A \in
\mathcal{L}(E; F)$ and $B\in \mathcal{L}(F;E)$ such that $B\circ A = I$. $E$ is said to have an unconditional finite dimensional expansion of the identity if there is a sequence of bounded linear operators $A_{n}:E\rightarrow E$ of finite rank, such that for $x\in E$
$$\sum_{n=1}^{\infty}A_{n}(x)=x$$
unconditionally.

We will say that the series $\displaystyle\sum_{n=1}^{\infty}x_{n}$ of elements of $X$ is weakly unconditionally Cauchy if $\displaystyle\sum_{n=1}^{\infty}|x^{\prime}(x_{n})|<\infty$ for all $x^{\prime}\in X^{\prime}$ or, equivalently if
$$\sup\bigg\{\bigg\|\sum_{n\in F}x_{n}\bigg\| ; F\subset \mathbb{N}, F finite\bigg\}<\infty.$$

A sequence $(x_{n})\subset E$ is a semi-normalized basic sequence if $(x_{n})$ is a Schauder basis for the closed subspace $M =
\overline{[x_{n} : n\in \mathbb{N}]}$, and moreover there are constant $a$ and $b$ such that $0< a<\|x_{n}\|<b$ for all $n\in\mathbb{N}$.
We denote by $(e_{n})$ the canonical basis of $c_{0}$. If $\Sigma$ is an algebra of subsets of a set $\Omega$, then a finitely additive vector
measure $\mu:\Sigma\rightarrow E$ is said to be strongly additive if the series $\displaystyle\sum_{n=1}^{\infty}\mu(A_{n})$ converges in norm for each
sequence $(A_{n})$ of pairwise disjoint members of $\Sigma$.
The Diestel-Faires theorem (see \cite[p.20, Theorem 2]{DIESTEL}) asserts that if $\Sigma$ is a $\sigma-$ algebra and $\mu:\Sigma\rightarrow E$
is not strongly additive, then $E$ contains an isomorphic copy of $\ell_{\infty}$.

\section{The main results}

The proof of our main results rests mainly on the following theorem of Ghenciu \cite{IOANA}, which generalizes results of several authors
\cite{EMA},\cite{LEW}, \cite{FEDER}.
\bigskip


\begin{theorem} \label{thm:(Teorema 10)}(\cite[Theorem 1]{IOANA})
Let $E$ and $F$ be Banach spaces, and let $G$ be a Banach space with an unconditional basis $(g_{n})$ and coordinate functionals $(g^{\prime}_{n})$.

\begin{enumerate}

\item [(a)] If there exist operators $R\in\mathcal{L}(G; F)$ and $S\in\mathcal{L}(E; G)$ such that $(R(g_{n}))$ is a seminormalized basic sequence
in $F$ and $(S^{\prime}(g^{\prime}_{n}))$ is not relatively compact in $E^{\prime}$, then $\mathcal{L}_{K}(E; F)$ is not complemented in
$\mathcal{L}(E; F)$.

\item [(b)] If there exist operators $R\in\mathcal{L}(G; F)$ and $S\in\mathcal{L}(E; G)$ such that $(R(g_{n}))$ is a seminormalized basic sequence
in $F$ and $(S^{\prime}(g^{\prime}_{n}))$ is not relatively weakly compact in $E^{\prime}$, then $\mathcal{L}_{wK}(E; F)$ is not complemented in
$\mathcal{L}(E; F)$.

\end{enumerate}
\end{theorem}

Emmanuele  \cite{EM} and John \cite{KA} independently proved that if $\mathcal{L}_{K}(E; F)$ contains a copy of $c_{0}$, then $\mathcal{L}_{K}(E; F)$ is not complemented in $\mathcal{L}(E; F)$ (see \cite[Theorem 2]{EM} and \cite[Theorem 1]{KA}). They also proved that if there exists a noncompact operator
$T\in \mathcal{L}(E; F)$ which factors through a Banach space with an unconditional basis, then $\mathcal{L}_{K}(E; F)$ contains a copy of $c_{0}$.
Clearly Theorem \ref{thm:(Teorema 10)} $(a)$ follows from these results.

\begin{theorem} \label{thm:(Teorema 11)}
Let $E$ and $F$ be Banach spaces, and let $G$ be a Banach space with an unconditional basis $(g_{n})$ and coordinate functionals $(g^{\prime}_{n})$.

\begin{enumerate}

\item [(a)] If there exist operators $R\in\mathcal{L}(G; F)$ and $S\in\mathcal{L}(E; G)$ such that $(R(g_{n}))$ is a seminormalized basic sequence
in $F$ and $(S^{\prime}(g^{\prime}_{n}))$ is not relatively compact in $E^{\prime}$, then $\mathcal{P}_{K}(^{n}E; F)$ is not complemented in
$\mathcal{P}(^{n}E; F)$ for every $n\in\mathbb{N}$.

\item [(b)] If there exist operators $R\in\mathcal{L}(G; F)$ and $S\in\mathcal{L}(E; G)$ such that $(R(g_{n}))$ is a seminormalized basic sequence
in $F$ and $(S^{\prime}(g^{\prime}_{n}))$ is not relatively weakly compact in $E^{\prime}$, then $\mathcal{P}_{wK}(^{n}E; F)$ is not complemented in
$\mathcal{P}(^{n}E; F)$ for every $n\in\mathbb{N}$.

\end{enumerate}
\end{theorem}

\begin{proof}

$(a)$ The case $n=1$ follows from Theorem \ref{thm:(Teorema 10)} $(a)$. If $n\in\mathbb{N}$, then by a result of Ryan \cite{RYAN} there exists an isomorphism
$$P\in\mathcal{P}(^{n}E; F)\rightarrow T_{P}\in\mathcal{L}(\hat{\otimes}_{n,s,\pi}E; F).$$
Furthermore $P\in \mathcal{P}_{K}(^{n}E; F)$ if and only if $T_{P}\in\mathcal{L}_{K}(\hat{\otimes}_{n,s,\pi}E; F)$.
Suppose that $\mathcal{P}_{K}(^{n}E; F)$ is complemented in $\mathcal{P}(^{n}E; F)$. Then $\mathcal{L}_{K}(\hat{\otimes}_{n,s,\pi}E; F)$ is complemented
in $\mathcal{L}(\hat{\otimes}_{n,s,\pi}E; F)$. Let $\pi:\mathcal{L}(\hat{\otimes}_{n,s,\pi}E; F)\rightarrow \mathcal{L}_{K}(\hat{\otimes}_{n,s,\pi}E; F)$
be a projection. By a result of Blasco \cite[Theorem 3]{BLASCO} $E$ is isomorphic to a complemented subspace of $\hat{\otimes}_{n,s,\pi}E$. Hence there exist operators $A\in\mathcal{L}(E;\hat{\otimes}_{n,s,\pi}E)$ and $B\in\mathcal{L}(\hat{\otimes}_{n,s,\pi}E; E)$ such that $B\circ A=I$. Consider the operator
$$\rho: T\in\mathcal{L}(E;F)\rightarrow \pi(T\circ B)\circ A\in \mathcal{L}_{K}(E; F).$$
If $T\in\mathcal{L}_{K}(E; F)$, then $T\circ B\in\mathcal{L}_{K}(\hat{\otimes}_{n,s,\pi}E; F)$ and therefore $\pi(T\circ B)\circ A=T\circ B\circ A=T$.
Thus $\rho: \mathcal{L}(E;F)\rightarrow \mathcal{L}_{K}(E; F)$ is a projection, contradicting the case $n=1$.

$(b)$ The proof of $(b)$ is almost identical to the proof of $(a)$, but using that $P\in \mathcal{P}_{wK}(^{n}E; F)$ if and only if $T_{P}\in \mathcal{L}_{wK}(\hat{\otimes}_{n,s,\pi}E; F)$, a result which is also due to Ryan \cite{RYAN}.
\end{proof}

\begin{theorem} \label{thm:(Teorema 12)}
Let $E$ and $F$ be Banach spaces, and let $G$ be a Banach space with an unconditional basis $(g_{n})$ and coordinate functionals $(g^{\prime}_{n})$.
If there exist operators $R\in \mathcal{L}(G; F)$ and $S\in\mathcal{L}(E; G)$ such that $(R(g_{n}))$ is a seminormalized basic sequence
in $F$ and $(S^{\prime}(g^{\prime}_{n}))$ is not relatively compact in $E^{\prime}$, then $\mathcal{P}_{w}(^{n}E; F)$ is not complemented in
$\mathcal{P}(^{n}E; F)$ for every $n\in \mathbb{N}$.
\end{theorem}

\begin{proof}
The method of proof of Theorem \ref{thm:(Teorema 11)} does not work here, since it is not true in general that $P\in\mathcal{P}_{w}(^{n}E; F)$ if and only if
$T_{P}\in \mathcal{L}_{w}(\hat{\otimes}_{n,s,\pi}E; F)$. Thus we have to proceed differently.
It follows from results of Aron and Prolla \cite{ARON} and Aron, Hervés and Valdivia \cite{VALDIVIA} that $\mathcal{P}_{w}(^{n}E; F)\subset \mathcal{P}_{K}(^{n}E; F)$ for every $n\in \mathbb{N}$, and it is easy to see that $\mathcal{P}_{w}(^{n}E; F)=\mathcal{P}_{K}(^{n}E; F)$ when $n=1$.
Thus the case $n=1$ follows from Theorem \ref{thm:(Teorema 10)} $(a)$. To prove the theorem by induction on $n$ it suffices to prove that
if $\mathcal{P}_{w}(^{n+1}E; F)$ is complemented in $\mathcal{P}(^{n+1}E; F)$, then $\mathcal{P}_{w}(^{n}E; F)$ is complemented in
$\mathcal{P}(^{n}E; F)$. Aron and Schottenloher \cite[Proposition 5.3]{AR} proved that $\mathcal{P}(^{n}E; F)$ is isomorphic to a complemented subspace of
$\mathcal{P}(^{n+1}E; F)$ when $F$ is the scalar field, but their proof works equally well when $F$ is an arbitrary Banach space. Thus there
exist operators $A\in\mathcal{L}(\mathcal{P}(^{n}E; F); \mathcal{P}(^{n+1}E; F))$ and $B\in\mathcal{L}(\mathcal{P}(^{n+1}E; F);\mathcal{P}(^{n}E; F))$
such that $B\circ A=I$. The operator $A$ is of the form
$$A(P)(x)=\varphi_{0}(x)P(x)$$
for every $P\in\mathcal{P}(^{n}E; F)$ and $x\in E$, where $\varphi_{0}\in E^{\prime}$ verifies that $\|\varphi_{0}\|=1=\varphi_{0}(x_{0})$, where
$x_{0}\in E$ and $\|x_{0}\|=1$. It is clear that if $P\in \mathcal{P}_{w}(^{n}E; F)$, then $A(P)\in \mathcal{P}_{w}(^{n+1}E; F)$.
Let us assume that $\mathcal{P}_{w}(^{n+1}E; F)$ is complemented in $\mathcal{P}(^{n+1}E; F)$, and let $\pi: \mathcal{P}(^{n+1}E; F)\rightarrow \mathcal{P}_{w}(^{n+1}E; F)$ be a projection. Consider the operator
$$\rho=B\circ \pi\circ A: \mathcal{P}(^{n}E; F)\rightarrow \mathcal{P}_{w}(^{n}E; F).$$
If $P\in \mathcal{P}_{w}(^{n}E; F)$, then $A(P)\in \mathcal{P}_{w}(^{n+1}E; F)$, and therefore
$$\rho(P)=B\circ \pi\circ A(P)=B\circ A(P)=P.$$
Thus $\rho:\mathcal{P}(^{n}E; F)\rightarrow \mathcal{P}_{w}(^{n}E; F)$ is a projection, and therefore $\mathcal{P}_{w}(^{n}E; F)$ is complemented
in $\mathcal{P}(^{n}E; F)$. This completes the proof.
\end{proof}

Ghenciu \cite{IOANA} derived as corollaries of Theorem \ref{thm:(Teorema 10)} results of several authors \cite{EMA}, \cite{LEW}, \cite{FEDER}, \cite{KALTON} and \cite{KA}. We now apply Theorems \ref{thm:(Teorema 11)} and \ref{thm:(Teorema 12)} to obtain polynomials versions of those corollaries.

\begin{corollary}\label{thm:(Teorema 13)}
If $F$ contains a copy of $c_{0}$ and $E^{\prime}$ contains a weak-star null sequence which is not weakly null, then $\mathcal{P}_{wK}(^{n}E; F)$ is
not complemented in $\mathcal{P}(^{n}E; F)$ for every $n\in \mathbb{N}$.
\end{corollary}

\begin{corollary}\label{thm:(Teorema 14)}
If $F$ contains a copy of $c_{0}$ and $E$ contains a complemented copy of $c_{0}$, then $\mathcal{P}_{wK}(^{n}E; F)$ is
not complemented in $\mathcal{P}(^{n}E; F)$ for every $n\in \mathbb{N}$.
\end{corollary}

\begin{corollary}\label{thm:(Teorema 15)}
If $F$ contains a copy of $\ell_{1}$ and $\mathcal{L}(E; \ell_{1})\neq \mathcal{L}_{K}(E; \ell_{1})$, then $\mathcal{P}_{wK}(^{n}E; F)$ is
not complemented in $\mathcal{P}(^{n}E; F)$ for every $n\in\mathbb{N}$.
\end{corollary}

When $n=1$ Corollaries \ref{thm:(Teorema 13)}, \ref{thm:(Teorema 14)} and \ref{thm:(Teorema 15)} correspond to \cite[Corollaries 2,3 and 5]{IOANA}.
Ghenciu derived those corollaries by observing that $E$ and $F$ satisfy the hypothesis of Theorem \ref{thm:(Teorema 10)} $(b)$. Since the hypothesis
of Theorem \ref{thm:(Teorema 10)} $(b)$ coincide with the hypothesis of Theorem \ref{thm:(Teorema 11)} $(b)$, we see that Corollaries \ref{thm:(Teorema 13)}, \ref{thm:(Teorema 14)} and \ref{thm:(Teorema 15)} follow from Theorem \ref{thm:(Teorema 11)} $(b)$.

\begin{corollary}\label{cor 33}
 If $F$ contains a copy of $c_{0}$ and $E$ is infinite dimensional, then:
\begin{enumerate}
\item [(a)] $\mathcal{P}_{K}(^{n}E; F)$ is not complemented in
$\mathcal{P}(^{n}E; F)$ for every $n\in \mathbb{N}$.

\item [(b)] $\mathcal{P}_{w}(^{n}E; F)$ is not complemented in
$\mathcal{P}(^{n}E; F)$ for every $n\in \mathbb{N}$.
\end{enumerate}
\end{corollary}

\begin{corollary}\label{cor 34}
 If $E$ contains a complemented copy of $\ell_{1}$ and $F$ is infinite dimensional, then:
\begin{enumerate}
\item [(a)] $\mathcal{P}_{K}(^{n}E; F)$ is not complemented in
$\mathcal{P}(^{n}E; F)$ for every $n\in \mathbb{N}$.

\item [(b)] $\mathcal{P}_{w}(^{n}E; F)$ is not complemented in
$\mathcal{P}(^{n}E; F)$ for every $n\in \mathbb{N}$.
\end{enumerate}
\end{corollary}

When $n=1$ Corollaries \ref{cor 33} and \ref{cor 34} correspond to \cite[Corollaries 4 and 6]{IOANA}.
Ghenciu derived those corollaries by observing that $E$ and $F$ satisfy the hypothesis of Theorem \ref{thm:(Teorema 10)} $(a)$. Since the hypothesis
of Theorem \ref{thm:(Teorema 10)} $(a)$ coincide with the hypothesis of Theorems \ref{thm:(Teorema 11)} $(a)$ and \ref{thm:(Teorema 12)}, we see that Corollaries \ref{cor 33} and \ref{cor 34} follow from Theorems \ref{thm:(Teorema 11)} $(a)$ and \ref{thm:(Teorema 12)}.

\begin{corollary}\label{cor 35}
 If $E$ contains a copy of $\ell_{1}$ and $F$ contains a copy of $\ell_{p}$, with $2\leq p<\infty$, then:
\begin{enumerate}
\item [(a)] $\mathcal{P}_{K}(^{n}E; F)$ is not complemented in
$\mathcal{P}(^{n}E; F)$ for every $n\in \mathbb{N}$.

\item [(b)] $\mathcal{P}_{w}(^{n}E; F)$ is not complemented in
$\mathcal{P}(^{n}E; F)$ for every $n\in \mathbb{N}$.
\end{enumerate}
\end{corollary}

\begin{proof}
We follow an argument of Emmanuele \cite[p. 334 ]{EM}. By a result of Pelczynski \cite{PELC}, if $E$ contains a copy of $\ell_{1}$,
then $E$ has a quotient isomorphic to $\ell_{2}$ (see also the proof of \cite{ARONLIBRO}). Let $S:E\rightarrow \ell_{2}$ be the quotient
mapping, and let $R:\ell_{2}\hookrightarrow \ell_{p}\subset F$ be the natural inclusion.
Since $S^{\prime}:\ell_{2}\rightarrow E^{\prime}$ is an embedding, the hypothesis of Theorems \ref{thm:(Teorema 11)} $(a)$ and \ref{thm:(Teorema 12)}
are clearly satisfied.
\end{proof}

\begin{proposition}\label{cor 366}
Let $E$ and $F$ be infinite dimensional Banach spaces. If $\mathcal{P}_{K} (^{n}E; F)$ contains a copy of $c_{0}$, then $\mathcal{P}_{K} (^{n}E; F)$ is not complemented in $\mathcal{P} (^{n}E; F)$.
\end{proposition}

\begin{proof}
By an aforementioned result of Ryan \cite{RYAN} we have that $P\in \mathcal{P}_{K}(^{n}E; F)$ if and only if $T_{P}\in\mathcal{L}_{K}(\hat{\otimes}_{n,s,\pi}E; F)$. Thus the result follows from \cite[Theorem 2]{EM} or \cite[Theorem 1]{KA}.

\end{proof}

The next proposition is a polynomial version of \cite[Theorem 2]{EM} and \cite[Theorem 1]{KA}. The proof is based in ideas of \cite[Corollary 11 ]{LEWIS}.
\begin{proposition}\label{cor 36}
Let $E$ be an infinite dimensional Banach space and $n>1$. If $\mathcal{P}_{w} (^{n}E; F)$ contains a copy of $c_{0}$, then $\mathcal{P}_{w} (^{n}E; F)$ is not complemented in $\mathcal{P} (^{n}E; F)$.
\end{proposition}

\begin{proof}
 By Corollary \ref{cor 33} and \cite[Lemma 5 ]{GONZALEZ M} we may suppose without loss of generality that $F$ contains no copy of $c_{0}$ and $E$ contains no complemented copy of $\ell_{1}$. By \cite[Theorem 3 ]{GONZALEZ M} $\mathcal{P}_{w} (^{n}E; F)$ contains no copy of $\ell_{\infty}$. Let $(P_{i})$ be a copy of the unit vector basis $(e_{i})$ of $c_{0}$ in $\mathcal{P}_{w} (^{n}E; F)$. Then $$\sup\bigg\{\bigg\|\sum_{i\in F}e_{i}\bigg\| ; F\subset\mathbb{N}, F finite\bigg\}=1.$$ By a result of Bessaga and Pelczynski \cite{BES} (see also \cite[p.44, Theorem 6]{DIESTELL}) the series $\displaystyle\sum_{i=1}^{\infty} e_{i}$ is weakly unconditionally Cauchy in $c_{0}$. This implies that the series
 $\displaystyle\sum_{i=1}^{\infty} P_{i}$ is weakly unconditionally Cauchy in $\mathcal{P}_{w} (^{n}E; F)$. For every $\varphi\in F^{\prime}$ and $x\in E$ we consider the continuous linear functional $$\psi: P\in\mathcal{P}_{w} (^{n}E; F)\rightarrow \varphi(P(x))\in \mathbb{C}.$$ Since the series
 $\displaystyle\sum_{i=1}^{\infty} P_{i}$ is weakly unconditionally Cauchy in $\mathcal{P}_{w} (^{n}E; F)$,
 $\displaystyle\sum_{i=1}^{\infty}|\psi( P_{i})|=\displaystyle\sum_{i=1}^{\infty}|\varphi( P_{i}(x))|<\infty$ for every $\varphi\in F^{\prime}$ and $x\in E$. This shows that
 $\displaystyle\sum_{i=1}^{\infty} P_{i}(x)$ is weakly unconditionally Cauchy in $F$ for each $x\in E$. Finally since $F$ contains no copy of $c_{0}$, an
 application of \cite[p.45, Theorem 8 ]{DIESTELL} shows that $\displaystyle\sum_{i=1}^{\infty} P_{i}(x)$ converges unconditionally in $F$ for each $x\in E$. Let $\mu: \wp(\mathbb{N})\rightarrow \mathcal{P} (^{n}E; F)$ be the finitely additive vector measure defined by $\mu(A)(x)=\displaystyle\sum_{i\in A}P_{i}(x)$ for each $x\in E$ and $A\subset \mathbb{N}$.
Suppose there is a projection $\pi:\mathcal{P}(^{n}E; F)\rightarrow \mathcal{P}_{w} (^{n}E; F)$. Then $\pi(P_{i})=P_{i}$ for each $i\in \mathbb{N}$.
 If the sequence $(\|P_{i}\|)$ does not converge to zero, then there is $\epsilon>0$ and a subsequence $(i_{k})$ of $\mathbb{N}$, such that $\|P_{i_{k}}\|>\epsilon$ for each $k\in \mathbb{N}$. But this implies that the measure $\pi\circ\mu:\wp(\mathbb{N})\rightarrow \mathcal{P}_{w}(^{n}E; F)$ is not strongly additive. Then the Diestel-Faires Theorem would imply that $\mathcal{P}_{w} (^{n}E; F)$ contains a copy of $\ell_{\infty}$. Therefore $\|P_{i}\|\rightarrow 0$, but this is absurd too, because $(P_{i})$ is a copy of $(e_{i})$. This complete the proof.
\end{proof}




The following theorem is a polynomial version of \cite[Theorem 2 ]{KA}.

\begin{theorem} \label{thm:(Teorema 17)}
Let $E$ and $F$ be Banach spaces and $P\in \mathcal{P}(^{n}E; F)$ such that $P\notin \mathcal{P}_{w}(^{n}E; F)$. Suposse that $P$ admits a factorization $P=Q\circ T$ through a Banach space $G$ with an unconditional finite dimensional expansion of the identity, where $T\in \mathcal{L}(E;G)$ and $Q\in \mathcal{P}(^{n}G;F)$. Then $\mathcal{P}_{w} (^{n}E;F)$ contains a copy of $c_{0}$ and thus $\mathcal{P}_{w} (^{n}E;F)$ is not complemented in
$\mathcal{P}(^{n}E;F)$.
\end{theorem}

\begin{proof}
The case $n=1$ follows from \cite[Theorem 2 ]{KA}.

Case $n> 1$: Since $G$ has an unconditional finite dimensional expansion of the identity, by \cite[Lemma 6 ]{GONZALEZ M} there is a sequence $(Q_{i})\subset \mathcal{P}_{w} (^{n}G; F)$ so that $Q(z)=\displaystyle\sum_{i=1}^{\infty}Q_{i}(z)$ unconditionally for each $z\in G$, hence $P(x)=\displaystyle\sum_{i=1}^{\infty}Q_{i}(T(x))$ unconditionally for each $x\in E$. Since $Q_{i}\in \mathcal{P}_{w} (^{n}G; F)$ for every $i\in \mathbb{N}$, it follows that $Q_{i}\circ T\in \mathcal{P}_{w} (^{n}E; F)$ for every $i\in \mathbb{N}$. By the uniform boundedness principle, we have $$\sup\bigg\{\bigg\|\sum_{i\in F}Q_{i}\circ T\bigg\| ; F\subset\mathbb{N}, F finite\bigg\}<\infty.$$
Again by \cite[p.44, Theorem 6]{DIESTELL} the series $\displaystyle\sum_{i=1}^{\infty}Q_{i}\circ T$ é weakly unconditionally Cauchy in $\mathcal{P}_{w} (^{n}E; F)$.
Since $P\notin \mathcal{P}_{w} (^{n}E; F)$, an application of \cite[p.45, Theorem 8]{DIESTELL} shows that $\mathcal{P}_{w} (^{n}E; F)$ contains a copy of $c_{0}$, and therefore by Proposition \ref{cor 36} $\mathcal{P}_{w} (^{n}E;F)$ is not complemented in $\mathcal{P}(^{n}E;F)$.
\end{proof}

\begin{corollary}\label{cor 300}
 Let $E$ and $F$ be Banach spaces, with $E$ infinite dimensional, and let $n>1$. If each $P\in \mathcal{P}(^{n}E; F)$ such that $P\notin \mathcal{P}_{w}(^{n}E; F)$ admits a factorization $P=Q\circ T$, where $T\in \mathcal{L}(E;G)$, $Q\in \mathcal{P}(^{n}G;F)$ and $G$ is a Banach space with an unconditional finite dimensional expansion of the identity, then the following conditions are equivalent:
 \begin{enumerate}

\item [(1)] $\mathcal{P}_{w} (^{n}E; F)$ contains a copy of $c_{0}$,
\item [($1^{\prime}$)] $\mathcal{P}_{K} (^{n}E; F)$ contains a copy of $c_{0}$,
\item [(2)] $\mathcal{P}_{w} (^{n}E; F)$ is not complemented in $\mathcal{P} (^{n}E; F)$,
\item [($2^{\prime}$)] $\mathcal{P}_{K} (^{n}E; F)$ is not complemented in $\mathcal{P} (^{n}E; F)$,
\item [(3)] $\mathcal{P}_{w} (^{n}E; F)\neq \mathcal{P} (^{n}E; F)$,
\item [($3^{\prime}$)] $\mathcal{P}_{K} (^{n}E; F)\neq \mathcal{P} (^{n}E; F)$,
\item [(4)] $\mathcal{P}(^{n}E; F)$ contains a copy of $c_{0}$,
\item [(5)] $\mathcal{P}(^{n}E; F)$ contains a copy of $\ell_{\infty}$.
\end{enumerate}
 \end{corollary}
 \begin{proof}

 $(1)\Rightarrow (2)$ by Proposition \ref{cor 36}.


$(2)\Rightarrow (3)$ is obvious.

$(3)\Rightarrow (1)$ by Theorem \ref{thm:(Teorema 17)}.

$(1)\Rightarrow (4)$ is obvious.

$(4)\Rightarrow (3)$ suppose $(4)$ holds and $(3)$ does not hold. Then $\mathcal{P}_{w} (^{n}E; F)=\mathcal{P}(^{n}E; F)\supset c_{0}$. Thus $(1)$ holds,
and therefore $(3)$ holds, a contradiction.

$(5)\Rightarrow (4)$ is obvious.

$(4)\Rightarrow (5)$ by a result of Ryan \cite{RYAN} $\mathcal{P}(^{n}E; F)$ is isometrically isomorphic to $\mathcal{L}(\widehat{\otimes}_{n,s,\pi}E; F)$.
Thus the result follows from (\cite[Remark 3 e) ]{KA} part $2\Rightarrow 3$).

Thus $(1)$, $(2)$, $(3)$, $(4)$ and $(5)$ are equivalent.

$(1)\Rightarrow(1^{\prime})$ is obvious.

$(1^{\prime})\Rightarrow (2^{\prime})$ by Proposition \ref{cor 366}.

$(2^{\prime})\Rightarrow (3^{\prime})$ is obvious.

$(3^{\prime})\Rightarrow (3)$ is obvious.

Since $(3)\Rightarrow (1)$ and $(1)\Rightarrow (1^{\prime})$, the proof of the corollary is complete.
\end{proof}

In particular if $E$ has an unconditional finite dimensional expansion of the identity we obtain \cite[Theorem 7]{GONZALEZ M}.
The assumptions of this corollary apply also if $F$ is a complemented subspace of a space with an unconditional basis.

\end{document}